\documentclass[11pt]{amsart}

\usepackage{latexsym}
\usepackage{amsfonts}

\def\rank{\mathop{\rm rank}\nolimits}

\newcommand{\field}[1]{\mathbb{#1}}
\newcommand{\C}{\field{C}}

\newcommand{\Q}{\field{Q}}

\newtheorem{defi}{Definition}[section]

\newtheorem{lem}[defi]{Lemma}
\newtheorem{theo}[defi]{Theorem}
\newtheorem{co}[defi]{Corollary}

\newtheorem{re}[defi]{Remark}

\title[Simple examples of affine manifolds with infinitely many exotic models]{Simple examples of affine manifolds with infinitely many exotic models}
\makeatletter

\@addtoreset{equation}{section}
\makeatother
\author{Zbigniew Jelonek}
\address[Z. Jelonek]{Instytut Matematyczny\\
Polska Akademia Nauk\\
\'Sniadeckich 8, 00-956 Warszawa, Poland}
\email{najelone@cyf-kr.edu.pl}

 \subjclass{14 R 10, 32 Q 99, }
\date{\today}

\begin{document}

\maketitle

\begin{abstract}
We give a simple general method of construction affine varieties
with infinitely many exotic models. In particular we show that for
every $d>1$ there exists a Stein manifold of dimension $d$,  which
 has uncountable many different structures
of affine variety.
\end{abstract}

\bibliographystyle{alpha}
\maketitle

\section{Introduction.}

Given any smooth complex affine variety $X$ , one can ask if there exists smooth affine
varieties $Y$ non isomorphic to $X$ but which are biholomorphic to $X$ when equipped
with their underlying structures of complex analytic manifold. When such exist, these
varieties $Y$ could be called exotic models  of $X$.

Examples of affine varieties with exotic models was found in
dimension two and three (see   \cite{f-m}, \cite{m-p}, \cite{p}).
Moreover, in \cite{j2} we showed that for every $n\ge 7$ there are
$n-$dimensional rational affine manifolds with exotic models. The
aim of this note is to give a simple general  method of
construction of such examples. In particular we show that such
examples do exist in any dimension $d>1$ (for $d=1$ it is easy to
see that such examples do not exist). Here we modify our idea from
\cite{j2} and we prove:

\begin{theo}
Let $V$ be a non-rational smooth affine curve. Then

(i) the affine surface $Y:=V\times \C$ has uncountably many
different exotic models.

(ii)  for every non $\C-$uniruled smooth affine variety $Z$ the
variety $Y\times Z$ has an exotic model. Moreover, if the group
$Aut(V\times Z)$ is at most countable, then  the Stein $n-$fold
$X\times Z$ has uncountably  many different structures of affine
variety.
\end{theo}

As a Corollary we have:

\begin{co}
Let $\Gamma_1,...,\Gamma_r$ be a finite family of smooth affine
non-rational curves ($r\ge 1$) and put $X=\C\times \prod^r_{i=1}
\Gamma_i$. Then the Stein manifold $X$ has uncountable many
different structures of affine variety. In particular for every
$d>1$ there exists a Stein manifold of dimension $d$,  which
 has uncountable many different structures
of affine variety.
\end{co}

Moreover we obtain:

\begin{theo}\label{22}
Let $V$ be a  smooth affine surface, which has a smooth completion
$\overline{V}$, such that $H^0(\overline{V},
K_{\overline{V}})\not=0.$ Then

(i) the affine fourfold $X:=V\times \C^2$ has infinitely many
different exotic models.

(ii) for every non $\C-$uniruled smooth affine variety $Z$ the
variety $X\times Z$ has an exotic model. Moreover, if the group
$Aut(V\times Z)$ is finite, then  the Stein $n-$fold $X\times Z$
has infinitely many different structures of affine variety.
\end{theo}

\begin{re}
{\rm In particular we can take as $V$ (above) any generic surface
$V\subset \C^3$ of degree $d\ge 4.$}
\end{re}

\section{Exotic models}

Let us recall the definition of a $\C$-uniruled variety which was
introduced in our paper \cite{jel}. First recall that  {\it a
polynomial curve} in $X$ is the image of the affine line $A^1(\C)$
under a non-constant morphism $\phi : A^1(\C) \to X.$ Now we have:

\begin{defi}
{\rm An affine  variety $X$ is said to be $\C$-{\it uniruled} if it
is of dimension $\geq 1$ and there exists a Zariski open,
non-empty subset $U$ of $X$ such that for every point $x\in U$
there is a polynomial curve in $X$ passing through $x.$}
\end{defi}

It is well-known, that if $X$ is projectively smooth variety,
which is $\C-$uniruled, then $H^0(\overline{X},
K_{\overline{X}})=0$, where $K_{\overline{X}}$ denotes the
canonical divisor of $\overline{X}.$ In the sequel we need the
following  theorem (compare with \cite{j1}):

\begin{theo}\label{tw1}
Let $X$ be a non-$\C$-uniruled smooth affine variety.  Let $\bf F,
G$ be  algebraic vector bundles on $X$ of $\rank$ $r.$ If the
total space of $\bf F$ is isomorphic to the total space of $\bf
G$, then $\bf F$ is isomorphic to $\bf \sigma^*G$ for some
automorphism $\sigma\in Aut(X).$
\end{theo}

\begin{proof}
Let $F$ denote the total space of $\bf F$ and $G$ the total space
of $\bf G.$ In what follows, we will identify $X$ with the zero
section of $\bf F$ and $\bf G$. Note that
$${\bf F}\cong TF|_X/TX, \ {\bf G}\cong TG|_X/TX.$$
Assume that there exists an isomorphism $\Phi : F\to G.$ Let $\pi
: \bf G \to X$ be the projection and take $f=\pi\circ\Phi.$ Since
the vector bundle $\bf F$ is locally trivial in the Zariski
topology,  Lemma 3.4 in \cite{j1}  shows that $\Phi({\bf F}_x)=
{\bf G}_{f(x)}$ for every $x\in X.$ Consequently, the mapping
$\sigma:=f|X : X \to X$ is an bijection. Moreover, it is easy to
check that for every $x\in X$ the derivative $d_x\sigma$ is an
isomorphism. Consequently the mapping $\sigma$ is an automorphism.
Take ${\bf G}'=\sigma^*{\bf G}.$ Let $\Sigma: G\to G'$ be the
induced isomorphism of total spaces (locally given as $U\times
\C^n\ni (x,v)\to (\sigma^{-1}(x),v)\in \sigma^{-1}(U)\times
\C^n).$ Replace $\Phi$ by $\Sigma\circ \Phi$ and ${\bf G}$ by
${\bf G}'$.

Now the mapping $\Phi_{|X} : X \in x \mapsto (x, t(x))\in  { G}$
is a section. Consider the isomorphism $\Psi : { G}\ni
(x,v)\mapsto (x, v-t(x))\in {\ G}.$ Again we can replace $\Phi$ by
$\Psi\circ \Phi$ to obtain $\Phi|X : X\times \{0\}\ni (x,0)\mapsto
(x,0)\in { G}.$ Hence we can assume that $\Phi$ transforms the
zero section into the zero section, and moreover it induces the
identity on the zero section. Hence $d\Phi(TX)=TX$ and the mapping
$$d\Phi :  TF|_X/TX\cong {\bf F}\to TG|_X/TX \cong {\bf G}$$
is an isomorphism. Consequently, the bundle $\bf F$ is isomorphic
to ${\bf G}$.
\end{proof}

Now we review some results about Stein manifolds. It is well-known
that a $n-$dimensional Stein manifold $X$ has the homotopy type of
a (real) $n-$dimensional CW complex ( see \cite{mil}). Complex
vector bundles on such CW complexes have tha following nice
property:

\begin{theo}\label{h}(\cite{hus}, p. 111)
Let $Y$ be a $r-$dimensional CW complex and let ${\bf F}$ be a
complex vector bundle on $Y$ of rank $k.$ If $r\le 2k-1$, then
$\bf F$ has a one dimensional trivial summand.
\end{theo}

Now we are ready to prove our first  result:

\begin{theo}
Let $V$ be a non-rational smooth affine curve. Then

(i) the affine surface $Y:=V\times \C$ has uncountably many
different exotic models.

(ii)  for every non $\C-$uniruled smooth affine variety $Z$ the
variety $Y\times Z$ has an exotic model. Moreover, if the group
$Aut(V\times Z)$ is at most countable, then  the Stein $n-$fold
$X\times Z$ has uncountably  many different structures of affine
variety.

\end{theo}

\begin{proof}
(i) Let $\overline{V}$ be a smooth compactification of $V$ and
$\{x_1,...,x_r\}=\overline{V}\setminus V.$ Then
$Pic(V)=Pic(\overline{V})/<x_1,...,x_r>.$ Since the subgroup
$<x_1,...,x_r>$ is countable and $Pic(\overline{V})$ not, we have
$Pic(V)\not=0$ - in fact this group is uncountable. Let ${\bf
L}\in Pic(V)$ be a non-zero line bundle. Hence it is algebraically
non-trivial. However, by Theorem \ref{h} $\bf L$ is
holomorphically trivial. Consequently the total space of any line
bundle ${\bf L}\in Pic(X)$ is biholomorphic to $Y.$

Note that the total space of every line bundle ${\bf L}\in Pic(V)$
determines one affine structure $Y_{\bf L}$ on $Y.$ Let $\rho$ be
the relation on $Pic(V)$ such that ${\bf L}$ is in a relation with
${\bf L}'$ if and only if there exists an automorphism $\sigma\in
Aut(V)$ : ${\bf L}'= \sigma^*{\bf L}.$ Since the group $Aut(V)$ is
finite, we see that the set $S:=Pic(V)/\rho$ is uncountable.
Denote the class of relation of ${\bf L}$ with respect to relation
$\rho$ by $[{\bf L}]$.

 Structures $Y_{\bf L}$ and $Y_{{\bf L}'}$ are
not isomorphic for $[{\bf L}]\not=[{\bf L}']$ by Theorem
\ref{tw1}. This means that there is at least $\#S$ different
affine structures on $Y.$

(ii) Let $\pi: V\times Z\to V$ be a projection. Take ${\bf
L}'=\pi^*({\bf L}).$ Then $\bf L$ is holomorphically trivial.
However, it is algebraically non-trivial. Indeed, take a point
$z\in Z. $ If we identify $V$ with $V\times \{z\}\subset V\times
Z$, then ${\bf L}'_{|V}={\bf L}.$ Now we can finish as above.

Note that the mapping $\pi^* : Pic(V)\ni {\bf L} \to \pi^*{\bf
L}\in Pic(V\times Z)$ is injective, hence the group $Pic(V\times
Z)$ is uncountable. If the group $Aut(V\times Z)$ is at most
countable, then the set $S'=A^2(V\times Z)/\rho$\ (where $\rho$ is
the relation as above) is uncountable. Now we can finish as above.
\end{proof}

\begin{co}
Let $\Gamma_1,...,\Gamma_r$ be a finite family of smooth affine
non-rational curves ($r\ge 1$) and put $X=\C\times \prod^r_{i=1}
\Gamma_i$. Then the Stein manifold $X$ has uncountable many
different structures of affine variety. In particular for every
$d>1$ there exists a Stein manifold of dimension $d$,  which
 has uncountable many different structures
of affine variety.
\end{co}

\begin{proof}
Let us note that $\overline{\kappa}(\Gamma_i)=1$ (where
$\overline{\kappa}$ denotes the logarithmic Kodaira dimension). We
have $\overline{\kappa}(\prod^r_{i=1} \Gamma_i)=\sum^r_{i=1}
\overline{\kappa}(\Gamma_i)=r$ (see \cite{iit}, Theorem 11.3).
Hence the variety $\prod^r_{i=1} \Gamma_i$ is of general type and
consequently it has a finite automorphisms group (see Theorem
11.12 in \cite{iit}).
\end{proof}

We show that  our method can be applied also to affine surfaces.
The following Lemma is well known:

\begin{lem}\label{chern}
Let $X$ be a smooth affine surface. Let $A^p(X)$ denotes the group
of codimension $p$-cycles modulo rational equivalence. Let $c_1\in
A^1(X), \ c_2\in A^2(X).$ Then there exists an algebraic vector
bundle $\bf F$ of rank $2$, such that $c_i({\bf F})=c_i$ for
$i=1,2,$ where $c_i(\bf F)$ is an $i^{th}$ Chern class of $\bf F$.
\end{lem}

\begin{proof}
Let $X=Spec(A).$ Let $\bf L$ be a line bundle which correspond to
$c_1.$ Moreover, let $A/I$ represent $c_2$, where $I$ is a product
of different maximal ideals. Then $Ext^1_A(I, L)$ is cyclic, where
$L$ is a module of sections of $\bf L.$ Following \cite{se} we get
an exact sequence
$$0\to L \to F \to I \to 0,$$ where $F$ is a projective module of
rank $2.$ If $\bf F$ is a vector bundle which correspond to $F$,
we get  $c_i({\bf F})=c_i$ for $i=1,2.$
\end{proof}

We have also:

\begin{lem}\label{wiazka}
Let $V$ be a  smooth affine surface, which has a smooth completion
$\overline{V}$, such that $H^0(\overline{V},
K_{\overline{V}})\not=0.$  For a given non zero class $c_2\in
A^2(X),$  there exists an algebraic vector bundle $\bf F$ of rank
two on $V$, which is algebraically non-trivial, but
holomorphically trivial and $c_2({\bf F})=c_2$.
\end{lem}

\begin{proof}
Note that $H^0(\overline{V}, K_{\overline{V}})\not=0.$  By Theorem
\ref{mum} we have $A^2(V)\not=0.$ Take nonzero $c_2\in A^2(V).$ By
Lemma \ref{chern} there is an algebraic vector bundle $\bf F$ of
rank $2$ such that $c_1({\bf F})=0$ and $c_2({\bf F})=c_2.$ In
particular $\bf F$ is algebraically non-trivial and it has trivial
determinant.

We show that $\bf F$ is holomorphically trivial. We will make use
of Grauert's theorem on the Oka principle for vector bundles which
says that on Stein spaces the holomorphic and topological
classifications coincide. Therefore we can use the topological
theory of complex vector bundles. Moreover, since every
$n-$dimensional Stein manifold has a homotopy type of a (real)
$n-$dimensional CW complex, if we study vector bundles on $V$, we
can assume that $V$ itself is a $2-$dimensional CW complex. In
particular by Theorem \ref{h} we have $\bf F= {\bf L}\oplus {\bf
E}^1$, where ${\bf E}^1$ denotes the trivial line bundle. Since
${\bf E}^1=\wedge^2 {\bf F}={\bf L}\otimes {\bf E}^1={\bf L}$ we
have ${\bf F}={\bf E}^1\oplus {\bf E}^1$ is holomorphically
trivial.
\end{proof}

Moreover, we need  the following result of Mumford and Roitman (
see \cite{mum}, \cite{roit}):

\begin{theo}\label{mum}
Let $X$ be an irreducible, proper, non-singular variety of
dimension $d$ over $\C$, such that $H^0(X; K_X) \not= 0,$ where
$K_X$ is the canonical divisor of $X.$ Then for any affine open
subset $V \subset X$, we have $A^d(V ) \not= 0.$
\end{theo}

Finally we have:

\begin{theo}\label{surface}
Let $V$ be a  smooth affine surface, which has a smooth completion
$\overline{V}$, such that $H^0(\overline{V},
K_{\overline{V}})\not=0.$ Then

(i) the affine fourfold $X:=V\times \C^2$ has infinitely many
different exotic models.

(ii) for every non $\C-$uniruled smooth affine variety $Z$ the
variety $X\times Z$ has an exotic model. Moreover, if the group
$Aut(V\times Z)$ is finite, then  the Stein $n-$fold $X\times Z$
has infinitely many different structures of affine variety.
\end{theo}

\begin{proof}
(i) Let ${\bf F}$ be a vector bundle as in Lemma \ref{wiazka}.
 This vector bundle  is algebraically non-trivial but
holomorphically trivial.

Note that the total space of every vector bundle ${\bf F}$ as
above, determines one affine structure $Y_{\bf F}$ on $Y.$ Let
$\rho$ be the relation on $A^2(V)$ such that ${a}$ is in a
relation with ${b}$ if and only if there exists an automorphism
$\sigma\in Aut(V)$ such that $a= \sigma^*b.$ Note that the group
$A^2(V)$ is infinite, because by \cite{roit2} we have
$A^2(V)\otimes \Q\not=0.$ Since the group $Aut(V)$ is finite (see
\cite{jel3}), we have that the set $S:=A^2(V)/\rho$ is infinite.
Denote the class of relation of $a\in A^2(V)$ with respect to
relation $\rho$ by $[a]$. Structures $Y_{\bf F}$ and $Y_{{\bf
F}'}$ are not isomorphic for $[c_2({\bf F})]\not=[c_2({\bf F}')]$
by Theorem \ref{tw1}. This means that there is at least $\#S$
different affine structures on $Y.$

(ii) Since $V$ is not $\C-$uniruled, then also $V\times Z$ is not
$\C-$uniruled. Let $\pi: V\times Z\to V$ be a projection. Take
${\bf G}=\pi^*({\bf F}).$ Then $\bf G$ is holomorphically trivial.
However, it is algebraically non-trivial. Indeed, take a point
$z\in Z. $ If we identify $V$ with $V\times \{z\}\subset V\times
Z$, then ${\bf G}_{|V}={\bf F}.$ Now we can finish as above.

Note that the mapping $\pi^* : A^2(V)\ni a \to \pi^*a\in
A^2(V\times Z)$ is injective, hence the group $A^2(V\times Z)$ is
infinite.   If the group $Aut(V\times Z)$ is finite, then the set
$S'=A^2(V\times Z)/\rho$\ (where $\rho$ is the relation as above)
is infinite. Now we can finish as above.
\end{proof}

\end{document}